\documentclass[a4paper,11pt]{amsart}
\pdfoutput=1
\usepackage[foot]{amsaddr}
\usepackage{amsfonts}
\usepackage{mathtools}
\usepackage[utf8]{inputenc}

\calclayout

\newcommand{\floor}[1]{\left\lfloor{#1}\right\rfloor}
\newcommand{\ceil}[1]{\left\lceil{#1}\right\rceil}

\newtheorem{theorem}{Theorem}[section]
\newtheorem{lemma}[theorem]{Lemma}
\newtheorem{proposition}[theorem]{Proposition}
\newtheorem{conjecture}[theorem]{Conjecture}
\newtheorem{corollary}[theorem]{Corollary}
\newtheorem{definition}[theorem]{Definition}
\newtheorem{observation}[theorem]{Observation}
\newtheorem{construction}[theorem]{Construction}

\begin{document}

\title{Minimum $k$-critical bipartite graphs}

\author{Sylwia Cichacz$^1$}
\author{Karol Suchan$^{1,2}$}\thanks{K.S. gratefully acknowledges financial support from Programa Regional STICAMSUD 19-STIC-05.}
\address{$^1$AGH University of Science and Technology, al. A. Mickiewicza 30, 30-059 Krakow, Poland}
\address{$^2$Universidad Diego Portales, Av. Ejército Libertador 441, 8370191 Santiago, Chile}
\email{cichacz@agh.edu.pl, karol.suchan@mail.udp.cl}

\begin{abstract}

We study the problem of Minimum $k$-Critical Bipartite Graph of order $(n,m)$ - M$k$CBG-$(n,m)$: to find a bipartite $G=(U,V;E)$, with $|U|=n$, $|V|=m$, and $n>m>1$, which is $k$-critical bipartite, and the tuple $(|E|, \Delta_U, \Delta_V)$, where $\Delta_U$ and $\Delta_V$ denote the maximum degree in $U$ and $V$, respectively, is lexicographically minimum over all such graphs. $G$ is $k$-critical bipartite if deleting at most $k=n-m$ vertices from $U$ creates $G'$ that has a complete matching, i.e., a matching of size $m$. We show that, if $m(n-m+1)/n$ is an integer, then a solution of the M$k$CBG-$(n,m)$ problem can be found among $(a,b)$-regular bipartite graphs of order $(n,m)$, with $a=m(n-m+1)/n$, and $b=n-m+1$. If $a=m-1$, then all $(a,b)$-regular bipartite graphs of order $(n,m)$ are $k$-critical bipartite. For $a<m-1$, it is not the case. We characterize the values of $n$, $m$, $a$, and $b$ that admit an $(a,b)$-regular bipartite graph of order $(n,m)$, with $b=n-m+1$, and give a simple construction that creates such a $k$-critical bipartite graph whenever possible. Our techniques are based on Hall's marriage theorem, elementary number theory, linear Diophantine equations, properties of integer functions and congruences, and equations involving them.
\end{abstract}

\keywords{fault-tolerance, interconnection network, bipartite graph, complete matching, algorithm, $k$-critical bipartite graph}

\subjclass[2010]{05C35, 05C70, 05C85, 68M10, 68M15, 11A05, 11Z05}

\maketitle

\section{Introduction}\label{sec:introduction}
\subsection{Design of fault tolerant networks}

Let $\Pi$ be a graph property that is {\em monotone}, i.e., it is preserved when adding edges to a graph. A graph $G$ is fault-tolerant with respect to $\Pi$ if the graph obtained from $G$ by removing a certain set of vertices and/or edges still has the property $\Pi$. Fault-tolerance has been studied with respect to different graph properties and under various fault scenarios: restrictions on the set of vertices and/or edges that can be removed while still maintaining the property $\Pi$. $\Pi$ being monotone motivates to focus on finding fault-tolerant graphs that are minimal with respect to taking subgraphs.

There has been particular interest in studying fault-tolerance with respect to properties defined as containing a subgraph isomorphic to a graph from a certain class. Given a family $\mathcal{H}$ of graphs and a positive integer $k$, a graph $G$ is called \emph{vertex $k$-fault-tolerant with respect to $\mathcal{H}$}, denoted by  $k$-FT$(\mathcal{H})$, if $G-S$ contains a subgraph isomorphic to some $H\in \mathcal{H}$, for every $S\subset V(G)$ with $|S|\leq k$. Note that in the literature, $k$-FT$(\mathcal{H})$ graphs are also called $(\mathcal{H},k)$-\emph{vertex stable} graphs. For a singleton $\mathcal{H}=\{H\}$, we write shortly $k$-FT$(H)$ instead of $k$-FT$(\{H\})$.
Clearly, $K_{q+k}$ is $k$-FT$({H})$ for every $q$-vertex graph $H$. So, there is a need to somehow measure the efficiency of $k$-FT$(\mathcal{H})$ graphs.

One option studied in the literature is to consider $k$-FT$(\mathcal{H})$ graphs having both small number of spare nodes and small maximum degree (to allow scalability of a network). Many papers concern the case when $\mathcal{H}=\{P_q\}$, i.e., the property consists in containing a path on $q$ vertices. Bruck, Cypher and Ho \cite{BCH} constructed $k$-FT$({P_q})$ graphs of order $q+k^2$ and maximum degree 4. Zhang's \cite{Zhang,Zhang2} constructions have $q+O(k\log^2k)$ vertices and maximum degree $O(1)$, and $q+O(k\log k)$ vertices and maximum degree $O(\log k)$, respectively. Alon and Chung \cite{refAloChu} gave, for $k=\Omega(q)$, a construction having  $q+O(k)$ vertices and maximum degree $O(1)$. Finally, Yamada and Ueno \cite{UY} constructed $k$-FT$(P_q)$ graphs with $q+O(k)$ vertices and maximum degree 3. 


A different quality measure of $k$-FT$(\mathcal{H})$ graphs was suggested by Hayes \cite{H}: a $k$-FT$(H)$ graph $G$ is called \emph{optimal} if $G$ has $|V(H)|+k$ vertices, and has the minimum number of edges among
all such graphs. A construction, having $q+k$ vertices and maximum degree $O(k\Delta(H))$, was given by  Ajtai, Alon, Bruck, Cypher, Ho, Naor and Szemer\'edi \cite{AABCHNS}. 


Yet another quality measure has been introduced by Ueno et al. \cite{ref_UenBagHaSch}, and independently by Dudek et al. in \cite{my}, where the authors were interested in $k$-FT$(H)$ graphs having as few edges as possible (disregarding the number of vertices). This topic has been widely studied \cite{CGNZ,CGZZ,ref_ErdGraSze,FTVW,FTVW2,Z,Z2,Z3}. Such a measure may be motivated by applications in sensor networks where sensors are much cheaper than the connections between them, and so their cost may be omitted.

\subsection{$k$-factor-critical graphs and $k$-extendable graphs}

In a graph $G=(V,E)$ of order $n$, a subset of edges $M$, $M \subset E$, is a \textit{matching} if for any $e, f \in M$, there is $e \cap f = \emptyset$. We say that the edges in $M$ are \textit{independent}. For a vertex $v \in V$ which is contained in one of the edges of $M$, we say that $v$ is {\em covered} by $M$. If $n$ is even and the size of $M$ is $n/2$, i.e., $M$ covers every vertex of $G$, then we say that $M$ is a \textit{perfect matching}.

A graph $G$ is called \textit{$k$-factor critical} or \textit{$k$-critical} if after deleting any $k$ vertices the remaining subgraph has a perfect matching. Thus, a graph $G$ of order $2n+k$ is \textit{$k$-critical} if and only if it is $k$-FT$(M)$, where $M$ is a matching of size $n$. This concept was first introduced and studied for $k = 2$ by Lov\'asz \cite{Lovasz}, under the term of \textit{bicriticality}. For $k>2$, it was introduced by Yu in 1993 \cite{Yu}, and independently by Favaron in 1996 \cite{Favaron}, this problem is also known under name $k$-\textit{matchable graphs} \cite{LY04}.  The idea of $k$-criticality is related to an older concept of $k$-extendability. 

Let $G$ be a graph of order $n$ with a perfect matching $M$, and let $0\leq k<n/2$ be a positive integer. A graph $G$ of even order $n \ge 2k+2$ is called \textit{$k$-extendable} if every matching of size $k$ in $G$ extends to (i.e., is a subset of) a perfect matching in $G$.  This concept was introduced by Plummer in 1980.

For $k =0$, $k$-extendable and $k$-factor-critical graphs are just graphs with a perfect matching. Moreover, we have the following theorem.

\begin{theorem}[\cite{ZWL}] 
  If $k\geq (|V (G)|+ 2)/4$, then a non-bipartite graph $G$ is $k$-extendable if and only if it is $2k$-critical.
\end{theorem}

It is straightforward that a bipartite graph cannot be $k$-critical for $k>0$. Nevertheless, we have the following result.

\begin{theorem}[\cite{Plummer}]\label{thm:extendability_ft}
    Let $G$ be a connected bipartite graph on $n$ vertices with bipartition $(U,V)$. Suppose $k$ is a positive integer
    such that $k \leq (n-2)/2$. Then the following are equivalent:
    \begin{enumerate}
        \item $G$ is $k$-extendable,
        \item $|U|=|V|$ and for each non-empty subset $X$ of $U$ such that $|X| \leq |U|-k$ and $|N(X)|\geq |X|+k$.
        \item For all $U' \subset U$ and $V' \subset \in V$, $|U'|=|V'|=k$, the graph $G'=G - U' - V'$ has a perfect matching.
    \end{enumerate}
\end{theorem}

Based on Theorem \ref{thm:extendability_ft}, $k$-extendability of bipartite graphs of order $2(n+k)$ can be seen as fault-tolerance for property $\Pi$ of containing a matching of size $n$, under attacks that consist in removing (at most) $k$ vertices from each color class. Indeed, this property is monotone by the following result.

\begin{corollary}[\cite{Plummer}]\label{co:add_edge}
    Suppose $k$ is a positive integer and $G=(U,V; E)$ is a $k$-extendable bipartite graph. For any $u\in U$ and $v \in V$, $G+uv$ also is $k$-extendable.
\end{corollary}

In line with the research on efficient design of fault-tolerant graphs, Zhang et al. \cite{ZZ12} presented a construction of $k$-extendable bipartite graphs with minimum number of edges and lowest possible maximum degrees.

On the algorithmic side, testing extendability in bipartite graphs reduces to testing connectivity in directed graphs based on the following result of Robertson et al.
\begin{theorem}[\cite{RST99}]
    Let $G=(U,V; E)$ be a connected bipartite graph, let $M$ be a perfect matching in $G$, and let $k \geq 1$ be an integer. Then $G$ is $k$-extendable if and only if $D(G,M)$ is strongly $k$-connected, where $D(G,M)$ is the directed graph obtained by directing every edge from $U$ to $V$, and contracting every edge of $M$.
\end{theorem}
Since connectivity can be efficiently tested in directed graphs (see, for example, the results of Henziger et al. \cite{H00}), so can the extendability of bipartite graphs.

\begin{corollary}\label{co:extendable_bipartite_poly}
    Given a bipartite graph $G$ and an integer $k$, it can be tested in polynomial time if $G$ is $k$-extendable.
\end{corollary}

In the literature also  some variations of edge-deletable extendable graphs and factor critical graphs are known \cite{PARK20116409,Plesnik1972,WANG20095242}.

\subsection{Vertex-fault tolerant design in bipartite graphs}

Our work is part of the line of research described above. For a motivation
of the particular problem we study, let us present a potential context of application.

Consider a network of $m$ sensing nodes and $n$ relay nodes.
Sensing nodes need to transmit their readings through the relay nodes,
using a pre-established infrastructure of links.
For each relay node, only one direct connection with a sensing node can be active at any time.
Relay nodes are faulty, but at any given time at most $n-m$ of them may be
unavailable, leaving the other $m$ relay nodes ready for transmission.

We want to establish a topology of links between sensing and relay nodes,
a bipartite graph, to guarantee that under any fault scenario (with at least $m$
relay nodes active), the $m$ sensing nodes can transmit their data through
distinct relay nodes.

Such an infrastructure needs to have at least $(n-m+1)m$
links. Indeed, suppose that the total number of links is smaller. Then
at least one sensing node $v$ is connected to at most $(n-m)$ distinct
relay nodes. And there is a fault scenario where precisely the
relay nodes linked with $v$ are inactive, in which case $v$ cannot
transmit its data.

We want to study topologies where not only the total number of links is low,
but also the maximum number of links per node is small (both on the relay
and sensing nodes sides).

Li and Nie formulated the problem in the language of graph theory,
they amended the definition of $k$-critical graph with respect
to bipartite graphs \cite{LN}. It requires that the $k$ vertices to be deleted lie in the
color class with more vertices. 

\begin{definition}\label{de:critical_bipartite}
A bipartite graph $G=(U, V;E)$ such that $k=|U|-|V|\geq 0$ is called a \textit{$k$-critical bipartite graph} if after deleting any $k$ vertices from the set $U$ the remaining subgraph has a perfect matching.
\end{definition}

Given a bipartite graph $G=(U, V;E)$ such that $k=|U|-|V|\geq 0$, let $\tilde{G}=(U, V \cup D;E \cup E^D)$ be the graph obtained from $G$ by adding the set $D$ of $k$ vertices to $V$ and making them adjacent to all vertices in $U$. Li and Nie showed the following theorem.

\begin{theorem}[\cite{LN}]\label{th:tilde}
$G$ is $k$-critical bipartite iff $\tilde{G}$ is $k$-extendable.
\end{theorem}

Therefore, by Corollary \ref{co:add_edge} and Theorem \ref{th:tilde}, we have the following corrolary.

\begin{corollary}
Suppose $k$ is a positive integer and $G=(U,V; E)$ is a $k$-critical bipartite graph. For any $u\in U$ and $v \in V$,
$G+uv$ also is $k$-critical bipartite.
\end{corollary}

So the property of being $k$-critical bipartite also is monotone, and the property of being $k$-critical bipartite graph of order $2n+k$ can be seen as fault-tolerance for property $\Pi$ of containing a matching of size $n$, under attacks that consist in removing (at most) $k$ vertices from the larger color class.

From the algorithmic point of view, by Corollary \ref{co:extendable_bipartite_poly} and Theorem \ref{th:tilde}, we obtain the following corollary.

\begin{corollary}\label{co:critical_bipartite_poly}
    Given a bipartite graph $G$ and an integer $k$, it can be tested in polynomial time if $G$ is $k$-extendable.
\end{corollary}


\subsection{Preliminaries}\label{ssec:preliminaries}

For the ease of reading, we will keep throughout the paper the ``semantics'' of the letters we use to denote objects. We will work with simple undirected bipartite graphs $G=(U,V; E)$, with $|U|=n$, $|V|=m$, and $|E|=e$. We will say that $G$ is of order $(n,m)$ and size $e$. For any $U' \subseteq U$, let us use the notation $H=G[U',V]$ to denote the subgraph induced by $(U',V)$ in $G$ and recall that $H=(U',V; F)$ is a bipartite graph with $U'$ and $V$ as its corresponding color classes. Similar to the use of upper case $U$ and $V$, with additional symbols where necessary, to denote the corresponding color classes and their subsets, we will use lowercase $u$ and $v$, also with additional symbols when needed, to denote vertices in $U$ and $V$, respectively. Let us use $\Delta_U$ and $\Delta_V$ to denote the maximum degree of a vertex in $U$ and $V$,
respectively. Similarly, let $\delta_U$ and $\delta_V$ denote the respective minimum degrees.

We say that $G$ is \textit{balanced} if $n=m$, and {\em unbalanced} otherwise. We say
that $G$ is \textit{biregular} if the degrees of the vertices in both color classes are constant,
and {\em irregular} otherwise. If $\delta_U = \Delta_U = a$ and $\delta_V = \Delta_V = b$, then we say that
$G$ is $(a,b)$-regular. Speaking of biregular graphs, we will use $a$ and $b$ to denote de degree of vertices in color classes $U$ and $V$, respectively. 

We will often compute the values of $a$ and $b$ based on the values of $n$ and $m$. In such cases, $n$ and $m$ are ``candidates'' to be the cardinalities of the color classes $U$ and $V$ of some bipartite graph $G=(U,V; E)$, and $a$ and $b$ are the ``candidates'' to be the corresponding degrees in case that $G$ happens to be biregular.

Given two sets $U'$ and $V'$, with $U' \subseteq U$ and $V' \subseteq V$, a subset of edges $M$, $M \subset E$, is a \textit{matching} from $U'$ to $V'$ if for every $e$, $e \in M$, we have $e=uv$, with $u \in U'$ and $v \in V'$, and for any $e, f \in M$, there is $e \cap f = \emptyset$. If $|M|=|U'|$, then $M$ is a \textit{complete matching} from $U'$ to $V'$. In other words, in a complete matching $M$, each vertex in $U'$ is incident to precisely one edge from $M$. A complete matching from $U'$ to $V'$ is perfect if $|U'|=|V'|$. 

Recall that Zhang et al. \cite{ZZ12} presented a construction of $k$-extendable bipartite graphs with minimum number of edges and lowest possible maximum degrees. Similarly to them, we are interested in such a construction for the case of $k$-critical bipartite graphs. 

We can now introduce the design problem that we study in this paper. Notice that if $m=1$, then the graph $K_{n,1}$ is the only bipartite graph of order $(n,1)$ that is $(n-1)$-critical bipartite. If $n=m$, then the graph $G=(U,V; E)$, where $|U|=|V|=m$, and $E$ is a perfect matching from $U$ to $V$ is $0$-critical bipartite. Moreover, $G$ is minimum with respect to the lexicographical order of $(|E|, \Delta_U, \Delta_V)$ among all $0$-critical bipartite graphs of order $(m,m)$. Our design problem can be seen as a generalization of the perfect matching construction to unbalanced bipartite graphs with $n>m>1$. 

We will work with the assumption of $n>m>1$ throughout the paper. The main object of study in this paper is given in the following definition.

\begin{definition}
Given positive integer values $n,m,k$ such that $n>m>1$ and $k=n-m$, the Minimum $k$-Critical Bipartite Graph problem for $(n,m)$ (M$k$CBG-$(n,m)$) is to find a bipartite graph $G=(U,V; E)$ of order $(n,m)$ that is $k$-critical bipartite, and minimum with respect to the lexicographical order of $(|E|, \Delta_U, \Delta_V)$.
\end{definition}

The optimality measure that we choose for our fault-tolerant design problem is in line with the work of Zhang et al. \cite{ZZ12}, and is motivated by scalability concerns that are natural in many applications.

A problem that is closely related to ours was presented by Perarnau and Petridis \cite{PP}. The authors studied the existence of perfect matchings in induced balanced subgraphs of random biregular bipartite graphs.
\begin{theorem}[\cite{PP}]

Let $k \in \mathbb{Q}^+$, $n \in \mathbb{Z}^+$ be arbitrarily large, and $b\in \{1,\ldots , n\}$, and suppose that $kn, kb \in \mathbb{Z}^+$,  with $kb \leq n$.

Furthermore, let $U$ and $V$ be sets of size $n$ and $kn$, respectively, and let $G$ be a graph taken uniformly at random among all $(kb,b)$-regular bipartite graphs on the vertex set $(U,V)$.
Take subsets $A \subset U$ and $B\subset V$ of size $kb$ and define $H := G[A, B]$ to be the subgraph induced in $G$ by vertex set $(A, B)$. Then
\begin{enumerate}
	\item No perfect matching exists in $H$ with high probability when\\
$\frac{kb^2}{n} -\log(kb) \underset{b \to \infty}{\longrightarrow}-\infty$ or when $b$ is a
constant.
\item A perfect matching exists in $H$ with high probability when\\ $\frac{kb^2}{n} -\log(kb) \underset{b \to \infty}{\longrightarrow}\infty$.
\end{enumerate}
\end{theorem}

Notice that their result is on balanced subgraphs of order $(kb, kb)$, and our work is
on subgraphs of order $(kn, kn)$, in graphs where $n > kn > 1$ and $b = n (1-k) + 1$.
So, in our setting, $kn > kb$.


The remainder of this paper is organized as follows. Section \ref{sec:applications} presents some relations with other areas of research where our work could find applications. Section \ref{sec:biregular} presents some constructions of biregular bipartite graphs that we use in the following sections.  Section \ref{sec:biregular_min} presents some particular properties of biregular bipartite graphs with $b = n - m + 1$. We also describe efficient ways to compute the values of the parameters for which $(a,b)$-regular bipartite graphs of order $(n,m)$ exist, which may be useful for analysis and applications in which generating such graphs is needed. Section \ref{sec:k-critical} presents some results on the construction of k-critical bipartite biregular graphs, with Theorem \ref{thm:main_positive} being the main contribution. The paper concludes with Section \ref{sec:conclusions}.

\section{Possible applications}\label{sec:applications}


Design of fault-tolerant bipartite graphs has potential applications in
the design of flexible processes, where there are $n$ different request
types and $m$ servers that should process them (see, for example, the work
of Chou et al. \cite{Chou} for a review of the topic).

From this point of view, our problem could model systems that work in a multi-period
setting, where an initial compatibility infrastructure has to be installed
(a bipartite graph between request types and servers) and then, at any time period,
there are at most $m$ different types of requests that have to be served.
Moreover, our restriction of assigning a distinct server to each request type would
correspond to systems where intra-period changes of set-up are not viable,
for example, due to a high set-up cost.

Up to now, the process flexibility literature has been focused on modeling systems where a server can process different kinds of compatible requests within the same time period (assigning to each of them a fraction of its time), therefore admitting fractional solutions. Moreover, only balanced systems, i.e., with $n = m$, used to be considered. However, the case of unbalanced systems has recently started to gain more interest (see, Deng et al. \cite{Deng} and Henao et al. \cite{Henao} for examples). And we hope that, with adequate tools, the case of systems that require the assignment to be exclusive within each time period, restricting the solutions to be integral, could be investigated.

A related line of research is not to design the smallest fault tolerant graph
(in terms of any metric, like the ones given above), but to analyze the level of
fault tolerance assured by prescribed topologies \cite{CH,PLK}. This topic is of particular
interest for algorithm design in high performance computing. Supercomputers are
comprised of many processing nodes (with some local memory) that use an interconnection
network to communicate during the execution of distributed algorithms. An algorithm
delegates computational tasks to different nodes, and uses some logical topology for
its message passing. This logical topology has to be somehow embedded in the
physical interconnection network provided by the supercomputer. So it is of practical
interest to study if the message passing topologies most common in algorithm
design (like cycles, and trees of certain types) can still be embedded
in interconnection topologies provided by supercomputers (often similar to hypercubes)
when the system presents some faults \cite{CH}.

Issues of graph modification in order to eliminate the existence of certain substructures
are also considered in studies on network interdiction (see \cite{S20} for a recent review).
Work on matching interdiction problems was initiated by Boros et al. in \cite{B06}.
They defined a {\em minimal blocker} to be an inclusion minimal set of edges in a bipartite
graph $G$ the removal of which leaves no perfect matching in $G$.
In \cite{Z09}, Zenklusen et al. introduced the concept
of a {\em $d$-blocker}, a set of edges the removal of which decreases the cardinality
of the maximum matching by at least $d$, and studied minimum cardinality d-blockers
in several graph classes. They showed that already the problem of deciding if
a given bipartite graph $G$ has a $d$-blocker of size at most $k$
is NP-Complete for any positive integer $d$. Moveover, they gave a construction
of minimum $d$-blockers for regular bipartite graphs.
In \cite{Z10}, Zenklusen expanded
the study to weighted graphs, defining the {\em edge interdiction problem} in a graph
with two functions, value and cost, on the edge set. It consists in finding a set of edges
of total cost bounded by a budget $B$ the removal of which decreases the value of the
maximum matching by at least $D$. In a similar way, he also defined the {\em vertex interdiction problem}
in a graph where the cost function is defined on the vertex set and we remove vertices.
He showed that the vertex inderdiction problem is NP-hard in bipartite graphs
even with unit edge values and vertex costs that are polynomially bounded in
the size of the graph, or with unit vertex costs and edge values bounded by a constant.
On the other hand, the vertex inderdiction problem can be solved in polynomial time
in bipartite graphs with unit edge values and unit vertex costs.

Most results on matching interdiction in bipartite graphs consider
eliminating (faults of) vertices from both color clases, but Laroche et al. \cite{L14}
considered the problem of eliminating vertices only from the larger color class
in the context of {\em Robust Nurse Assignment Problem}.
In their setting, the larger color class of a bipartite graph corresponds to
the set of nurses and the smaller color class to the set of roles. The edges represent
the skill sets of the nurses (the set of roles that each of them can assume).
The problem asks for the maximum number of absent nurses that still permits to assign each
job to a different nurse. They show that this special case of vertex interdiction
problem can be solved in polynomial time.

Adjiashvili et al. \cite{A16} expanded the work on robust assignments
to the field of robust optimization. Given a bipartite graph $G(U,V; E)$, $|U| \geq |V|$, with costs
on edges, the {\em Edge-Robust Assignment Problem} is to find a minimum cost set of edges $E'\subset E$ such that
under any scenario of a set $F$ of edges that fail, taken from a fixed set $\mathcal{F}$
of possible fault scenarios, the remaining graph $G(U,V; E' \setminus F)$ contains
a complete matching. They give some negative and positive complexity results for
special cases of only single-edge-fault scenarios. In \cite{A17}, they complemented
their results by considering the {\em Vertex-Robust Assignment Problem}. Here,
the costs are on the vertices of $U$, and the problem is to find a minimum
cost subset of vertices $U'$, $U' \subseteq U$, such that under any fault scenario
$F$ of vertices that fail, taken from a fixed set $\mathcal{F}$ of possible fault scenarios,
the remaining graph $G[U'\setminus F,V] $ contains a complete matching. In general,
$\mathcal{F}$ can be any family of subsets of $U$. In \cite{A17}, the authors study
the complexity of problems considering only single-vertex-fault scenarios.

The problem we study is also related to the work of Assadi and Bernstein on graph sparsification for the
edge-fault tolerant approximate maximum matching problem \cite{AB}.
They offer a polynomial time algorithm that, given any graph $G=(V,E)$, $\varepsilon > 0$,
and $f \geq 0$, computes a subgraph $H=(V,E')$ of $G$ such that for any
set of edges $F$, $|F| = f$, the maximum matching in $(V,E \setminus F)$ is at most
$3/2+\varepsilon$ times larger than the maximum matching in $(V,E' \setminus F)$. Moreover,
$H$ hast $O(f+n)$ edges. To our knowledge, this is the only non-trivial algorithm
known for edge-fault tolerant approximate maximum matching sparsification, and there are
no algorithms for the vertex-fault tolerant version. We believe that our results
might bring some insights valuable for this kind of challenges. 

\section{Constructions of biregular graphs}\label{sec:biregular}

In this section we present some constructions of biregular bipartite graphs that will be useful throughout the paper. Let us start with a lemma.

\begin{lemma}\label{le:xy}
Let $n,m,a,b$ be positive integers such that and $an=bm$. Let $c=\gcd(m,n)$ and $d=\gcd(a,b)$. Then there exist positive coprime integers $x$, $y$ such that $n = x c$, $m = y c$, $a = y d$, $b = x d$. 
\label{podzielniki}
\end{lemma}
\begin{proof}

Let $n = n' \gcd(n,m)$, $m = m'\gcd(n,m)$, $a = a' \gcd(a,b)$, $b = b' \gcd(a,b)$.
Since $n a = b m$, we have $\frac{m}{n}=\frac{a}{b}$ and so $\frac{m'}{n'}=\frac{a'}{b'}$.
Since both are irreducible fractions, there is $x := n' = b'$ and $y := m' = a'$, and they are coprime. So we have
$n = x \gcd(n,m)$, $b = x \gcd(a,b)$, $m = y \gcd(n,m)$, $a = y \gcd(a,b)$.

\end{proof}

For a  positive integer $o$, let $[o]=\{0,1,\ldots,o-1\}$. As usual, the notation $x = y\pmod n$  indicates a congruence relation, whereas the notation $y\bmod n$ is a binary operation which returns the remainder. The values of $x$, $y$, $c$, and $d$, as they appear in Lemma \ref{le:xy} are relevant throughout the paper. Like for $n$ and $m$ being the ``candidates'' for the cardinalities of the color classes $U$ and $V$, and $a$ and $b$, being the ``candidates'' for the corresponding degrees in a graph $G$ under consideration, we will consider the objects denoted by the letters $x$, $y$, $c$, and $d$ to be the ``candidates'' for the parameters described here.

\begin{construction}\label{cons1}
Let $n,m,a,b$ be positive integers such that $an=bm$. 
Let $c=\gcd(m,n)$ and $d=\gcd(a,b)$. Define $G_1=(U,V; E)$ as a bipartite graph having color classes
$U=\{u_i \mid i \in [n]\}$, $V=\{v_j \mid j \in [m]\}$, and edges
$E = \left\{ (u_i, v_{(j+\alpha) \bmod{m}}) \mid i \in [n], \alpha \in [a], j=\floor{\frac{i}{x}}y\right\}$.
\end{construction}

It is easy to check that, the graph $G_1=(U,V; E)$ can also be constructed in the following way. Namely, let $G^{\prime}$ be a $d$-regular bipartite graph having color classes
$U^{\prime}=\{u_i \mid i \in [c]\}$ and $V^{\prime}=\{v_j \mid j \in [c]\}$ such that $E(G^{\prime})=\{u_iv_{(i+\delta)\bmod{c}},\; i \in [c], \delta \in [d] \}$.

We will construct now the graph $G_1=(U,V; E)$ by ``blowing up'' each vertex $u_i$ into $x=b/d=n/c$ vertices
$u_{i,\alpha}$, $\alpha \in [x]$, whereas each vertex $v_j$ into $y=a/d=m/c>1$ vertices $v_{j,\beta}$,
$\beta \in [y]$. Each edge from $G'$ will be substituted now by the corresponding complete bipartite graph $K_{x,y}$.




Let us present a similar construction, that also connects each vertex from $U$
with an interval of $a$ consecutive vertices from $V$, only that now the first
vertex is chosen in a slightly different way. We will start with some useful lemmas.

\begin{lemma} \label{resxy}

Let $\alpha\in[c]$. If $\ceil{i\frac{y}{x}}\bmod m = j$, then $\ceil{(i+\alpha x)\frac{y}{x}}\bmod m = (j+\alpha y) \bmod m$.

\end{lemma}
\begin{proof}
$\ceil{(i+\alpha x)\frac{y}{x}}\bmod m= (\alpha y+\ceil{i\frac{y}{x}})\bmod m=(\alpha y + j) \bmod m$.
\end{proof}

\begin{lemma}
Let $x$, $y$, $c$, $n=cx$, $m=cy$ be positive integers, with $y<x$ and $j \in [m]$.
Then the number of integer solutions to $\ceil{\frac{iy}{x}} = j \pmod{m}$
with respect to $i$, with $i \in [n]$, is equal to $\floor{j\frac{x}{y}} - \floor{(j-1)\frac{x}{y}}$.  Moreover:
\begin{itemize}
	\item $\floor{j\frac{x}{y}} - \floor{(j-1)\frac{x}{y}} = \floor{r\frac{x}{y}} - \floor{(r-1)\frac{x}{y}}$,
where $r = j \bmod y$.

\item For any interval of consecutive $y$ values of $j$, for $(x \bmod y)$ of them, there are $\ceil{\frac{x}{y}}$ solutions and, for the remaining $(y - x \bmod y)$, there are $\floor{\frac{x}{y}}$ solutions.

\item In general, the number of solutions is $\ceil{\frac{x}{y}}$ for $(n \bmod m)$, and $\floor{\frac{x}{y}}$ for $(m - n \bmod m)$ values of $j\in[m]$.
\end{itemize}
 \label{number_of_i_1}
\end{lemma}
\begin{proof}
We have $\ceil{\frac{iy}{x}} \leq \ceil{\frac{(n-1)y}{x}} = \ceil{\frac{cxy-y}{x}}
= cy + \ceil{-\frac{y}{x}} = m$. So $\ceil{\frac{iy}{x}} \bmod{m} = j$
means $\ceil{\frac{iy}{x}} = j$ if $j>0$. If $j=0$, then we have
$\ceil{\frac{iy}{x}} = 0$ or $\ceil{\frac{iy}{x}} = m$.

First, consider the case $j>0$. It means that $(j-1) < \frac{iy}{x}$ and
$\frac{iy}{x} \leq j$. So the equation is satisfied iff $i$ belongs to the
interval $\left](j-1)\frac{x}{y}, j\frac{x}{y}\right]$. The number of integers
it contains is equal to $\floor{j\frac{x}{y}} - \floor{(j-1)\frac{x}{y}}$
(for example, see Chapter 3 of the book by Graham, Knuth and Patashnik \cite{GKP}).
And they all belong to $[n]$.

Now, let $j=0$. The only solution to $\ceil{\frac{iy}{x}} = 0$ with $i \in [n]$ is $i=0$.
On the other hand,  $\ceil{\frac{iy}{x}} = m$, without additional restrictions on $i$,
is satisfied by all integers inside the interval $\left](m-1)\frac{x}{y}, m\frac{x}{y}\right]$.
The only one of them not contained in $[n]$ is $m\frac{x}{y} = cy\frac{x}{y} = cx = n$.
So the number of integer solutions to $\ceil{\frac{iy}{x}} \pmod{m} = 0$ with
$i \in [n]$ is equal to the number of integer solutions of $\ceil{\frac{iy}{x}} = m$.

Let us analyze the behavior of the number of solutions to $\ceil{\frac{iy}{x}} = j$
with respect to the residue of $j \pmod{y}$. Let $j = l y + r$ for some $0 \leq r < y$
and integer $l$. Then $\floor{j\frac{x}{y}} - \floor{(j-1)\frac{x}{y}}
= \floor{\frac{lyx + rx}{y}} - \floor{\frac{lyx + rx - x}{y}}
= \floor{r\frac{x}{y}} - \floor{(r-1)\frac{x}{y}}$.

Let us now consider $r\in[y]$ and suppose that $\ceil{i\frac{x}{y}}=r$, then  $i\in[x]$ by Lemma~\ref{resxy}.
If $x=qy+p$ for $p\in[y]$, then the number of integer solutions to $\ceil{i\frac{y}{qy+p}}=r$
with respect to $i$, with $i \in [qy+p]$, is equal to either $q=\floor{\frac{x}{y}}$ or $q+1=\ceil{\frac{x}{y}}$.
Because  $i\in[x]$, the number of solutions is $\ceil{\frac{x}{y}}$ for $p=x\bmod y$ values of $r\in[y]$.
For the other $(y - x \bmod y)$ values of $r$, there are $\floor{\frac{x}{y}}$ solutions. Since
the number of solutions depends only on $(j \bmod y)$, the second bullet holds.


Finally, the number of solutions is $\ceil{\frac{x}{y}}$ for $c\cdot( x \bmod y)=  (cx \bmod cy)  = (n \bmod m)$ values of for $j\in[m]$.
\end{proof}

\begin{lemma}\label{consecutive}
Let $x$, $y$, $d$, $a=d y$, $b=d x$,  $c$, $n=cx$, $m=cy$, be positive integers, with $y<x$, $d < c$, and $j \in [m]$.
Then the number of integers $i$, with $i \in [n]$, such that $\ceil{\frac{iy}{x}} \pmod{m} = j$
with $j \in \{l-(a-1),\dots,l\} \pmod{m}$ for some integer $l$, $l \in [m]$, is equal to $b$.\label{number_of_i_2}
\end{lemma}
\begin{proof}
By Lemma \ref{number_of_i_1}, the number of solutions to $\ceil{\frac{iy}{x}} \pmod{m} = j$
depends only on the residue $j \pmod{y}$. On the other hand, by the conditions of the Lemma,
there is $a+y \leq m$.  So, without loss of generality, we may assume that $m-1 \geq l \geq m - y$
and $l-(a-1) \geq 0$.

By Lemma \ref{number_of_i_1}, number of integers $i$, as stated in the thesis, is equal to
$\sum_{j=l-(a-1)}^{l} (\floor{j\frac{x}{y}} - \floor{(j-1)\frac{x}{y}})$.
By the telescopic property, it is equal to
$\floor {l\frac{x}{y}}-\floor {(l-a)\frac{x}{y}} =
 \floor {l\frac{x}{y}}-\floor {l\frac{x}{y}-dy\frac{x}{y}} =
 \floor {l\frac{x}{y}}-\floor {l\frac{x}{y}}+dx = dx = b$.


\end{proof}

\begin{construction}\label{cons3}
Let $n,m,a,b$ be positive integers such that $1<m<n$ and $an=bm$.
Define $G_2=(U,V; E)$ as a bipartite graph having color classes
$U=\{u_i \mid i \in [n]\}$, $V=\{v_j \mid j \in [m]\}$, and edges
$E = \left\{ (u_i, v_{(j+\alpha) \bmod{m}}) \mid i \in [n], \alpha \in [a], j=\ceil{\frac{i y}{x}}\right\}$.
\end{construction}
We will show now that the graph  $G_2=(U,V; E)$ is $(a,b)$-regular.
\begin{observation}Let  $G_2=(U,V; E)$ be a graph of order $(n,m)$ given by Construction~\ref{cons3}, then $\delta_U = \Delta_U = a$ and $\delta_V = \Delta_V = b$.\label{degrees}
\end{observation}
\begin{proof}
One can  easily see that $G_2$ is of order $(n,m)$ and $\delta_U = \Delta_U = a$. Let us
show that we also have $\delta_V = \Delta_V = b$.
Notice that any vertex $v_l \in V$ is adjacent to all vertices $u_i$ such that
$\ceil{\frac{iy}{x}} = j$ with $j \in \{l-(a-1),\dots,l\} \pmod{m}$.
So, by Lemma \ref{number_of_i_2}, $d(v_l) = b$.
\end{proof}

\section{Biregular graphs with $b=n-m+1$}\label{sec:biregular_min}

Observe that, given integer values $n$ and $m$, $n > m > 1$, by Pigeonhole principle, we need to have $\delta_V\geq b = n-m+1$ for any $k$-critical bipartite graph $G=(U,V;E)$ of order $(n,m)$, where $k=n-m$. So it is interesting to study graphs where $\Delta_V = n-m+1$. In particular, if $a=m(n-m+1)/n$ is an integer, we consider the $(a,b)$-regular bipartite graphs of order $(n,m)$, with $b = n-m+1$. If the subset of them that are $k$-critical bipartite is not empty, then it is the set of solutions to the M$k$CBG-$(n,m)$ problem. 

In this section we study conditions under which $(a,b)$-regular bipartite graphs of order$(n,m)$, where $a = \frac{m (n-m+1)}{n}$ is integer and $b = n - m + 1$, exist. In the following section we will show that some of them indeed are $k$-critical bipartite.

Recall that for an $(a,b)$-regular bipartite graphs of order $(n,m)$, there are other relevant parameters: $x$, $y$, $c$, $d$, as described in Lema \ref{le:xy}. Even though the main result of this paper, Theorem \ref{thm:main_positive}, is based on a fixed quadruple of parameters $n$, $m$, $a$, $b$, here we present a wider vision on how such a quadruple can be completed, when only some of the parameters are fixed, and even based on the values of the parameters $x$, $y$, $c$, $d$ that are ``auxiliary'' for the Theorem \ref{thm:main_positive}. We believe that these results are interesting in their own right, and might find applications in studies on generating (random) biregular bipartite graphs.

Recall that we use the letters $n$, $m$, $a$, $b$, $x$, $y$, $c$, $d$ to represent values that are ``candidates'' to be the parameters of a $(a,b)$-regular bipartite graph of order $(n,m)$, where $c=\gcd(n, m)$, $d=\gcd(a,b)$, $n = x c$, $m = y c$, $a = y d$, $b = x d$. Especially in this section, only after we have verified all the conditions we can be sure that the corresponding graphs actually exist.

%
%


Sometimes it will be useful to rewrite the other variables in terms of $x$, $y$ and $c$. By Lemma \ref{le:xy}, and rewriting in terms of $x$, $y$, $c$, we get the following lemma. 

\begin{lemma}\label{rewrite_xyc}
Let $n,m$ be positive integers such that $1<m<n$, and $\frac{m (n-m+1)}{n}$ is integer. Let $a = \frac{m (n-m+1)}{n}$, $b = n - m + 1$, $c = \gcd(n,m)$, $d = \gcd(a,b)$, $x=\frac{n}{c}$, $y=\frac{m}{c}$. Then the following hold:
\[
\begin{array}{rcl}
d & = & \frac{c (x-y) + 1}{x} \\
a & = & y \frac{c (x-y) + 1}{x}\\
b & = & c (x-y) + 1\\
m & = & y c\\
n & = & x c\\
\end{array}
\]
\end{lemma}

%

The following lemma introduces another parameter, $p$, that is  useful in the analysis of biregular graphs with $b=n-m+1$, and describes some related properties.

\begin{lemma}\label{params_properties}
Let $n,m$ be positive integers such that $1<m<n$, and $\frac{m (n-m+1)}{n}$ is integer. Let $a = \frac{m (n-m+1)}{n}$, $b = n - m + 1$, $c = \gcd(n,m)$, $d = \gcd(a,b)$, $x=\frac{n}{c}$, $y=\frac{m}{c}$, and $p = c - d$. There is $p = \frac{m - 1}{x} = \frac{cy - 1}{x} = \frac{yd - 1}{x-y} = \frac{a - 1}{x-y} = \frac{m - a}{y}$.
Moreover $p > 0$, $c > d > 0$, $c >p$, $x > y > 0$, $x > x-y$, $n>m>2$, $b>a>1$, $n-x \geq b$, $m-y \geq a$,
$n-c \geq m \geq c$, $b-d \geq a \geq d$.
\end{lemma}
\begin{proof}
The first part follows from Lemma \ref{rewrite_xyc}, by simple rewriting of the terms. For the second part, first notice that by our choice of variables, $p$ is integer. Then, since $p = \frac{m - 1}{x}$, it must be non-negative. Finally, $p=0$ would imply $m=1$ - a contradiction with $m>1$. Since $p = \frac{a-1}{x-y}$, $p \geq 1$, $x-y \geq 1$, we have $a \geq 2$. Since $m-1 = px$, and $x \geq 2$, there is $m \geq 3$.
\end{proof}

In the following, we will use properties of the B\'ezout's Identity (a detailed treatment
can be found, for example, in \cite{AAC}).

\begin{theorem}[B\'ezout's Identity]
Let $\alpha$, $\beta$ and $\gamma$ be integers, $\alpha$, $\beta$ nonzero. Consider the linear Diophantine equation
\[\alpha \phi + \beta \psi = \gamma .\]
\begin{enumerate}
\item The equation is solvable for $\phi$ and $\psi$ in integers iff $\gamma' = \gcd(\alpha,\beta)$ divides $\gamma$.
\item If $(\phi_0, \psi_0)$ is a particular solution, then every integer solution is of the form:
\[\phi = \phi_0 + l\frac{\beta}{\gamma} , \psi = \psi_0 - l\frac{\alpha}{\gamma}\]
where $l$ is an integer.
\item If $\gamma = \gcd(\alpha, \beta)$ and $|\alpha|$ or $|\beta|$ is different from $1$, then a particular solution
$(\phi_0, \psi_0)$ can be found such that $|\phi_0| < |\beta|$ and $|\psi_0| < |\alpha|$, as the coefficients
obtained in the Extended Euclidean Algorithm.
\end{enumerate}
\label{bezout}
\end{theorem}

Given a linear Diophantine equation like in Theorem \ref{bezout}, we will say that $\phi$ and $\psi$
are {\em B\'ezout's coefficients} for $\alpha$ and $\beta$.

\begin{lemma}\label{cd_coprime}
Let $n,m$ be positive integers such that $1<m<n$, and $\frac{m (n-m+1)}{n}$ is integer. Let $a = \frac{m (n-m+1)}{n}$, $b = n - m + 1$, $c = \gcd(n,m)$, $d = \gcd(a,b)$, $x=\frac{n}{c}$, $y=\frac{m}{c}$. Then the identity $dx + (-c)(x-y) = 1$ holds and the pairs $(c, d)$, $(c,x)$ and $(x,x-y)$ are coprime.
\end{lemma}
\begin{proof}
The identity follows from rewriting of the relations presented in Lemma \ref{rewrite_xyc}. Coprimality follows from the identity and Theorem \ref{bezout}.
\end{proof}

\begin{proposition}\label{xy_existence}
For any integer pair $(x,y)$ with $0<y<x$ and $\gcd(x,y)=1$, there exist infinitely many positive integer pairs $(c,d)$ such that there exists an $(a,b)$-regular graph $G$ of order $(n,m)$, where $n = x c$, $m = y c$, $a = y d$, $b = x d$, with $b=n-m+1$. All such $(c_l,d_l)$ pairs can be computed as $(c_l,d_l) = (c_0 + l x, d_0 + l (x-y))$ for any positive integer $l$, where $d_0$ and $(-c_0)$ are B\'ezout coefficients for $x$ and $(x-y)$ in $dx + (-c)(x-y) = 1$.
\end{proposition}
\begin{proof}
Let $(x,y)$ be an integer pair with $0<y<x$ and $\gcd(x,y)=1$. Given any integer pair $(c,d)$, applying the formulas $n = x c$, $m = y c$, $a = y d$, $b = x d$ gives us integer values of $n$, $m$, $a$, $b$. Moreover, we have $a = y d = \frac{x d y c}{x c} = \frac{b m}{n}$, and so $a n = b m$. So the existence of the corresponding $(a,b)$-regular graph of order $(n,m)$, by Construction \ref{cons1}, is equivalent to satisfying the equation $b = n - m + 1$.

Rewriting $b = n - m + 1$, we obtain $x d = c x - c y + 1$, and $1 = d x - c (x-y)$. So, such a graph $G$ exists for any pair of integers $(c,d)$ such that $1 = d x - c (x-y)$. By Theorem \ref{bezout}, $d$ and $(-c)$ are B\'ezout coefficients for $x$ and $(x-y)$.

Notice that, since $x > 1$, the Extended Euclidean Algorithm will give us B\'ezout's coefficients $d_E$, $(-c_E)$ such that
$|d_E| < x-y$ and $|(-c_E)| < x$.

If both $d_E$ and $c_E$ are positive, since $x > x-y$, we have $d_E<c_E$ and $d_E - (x-y) < 0$.
So we can set $d_0=d_E$ and $c_0=c_E$, and all the other solutions can be constructed as $d_l = d_0 + l (x-y)$
and $c_l = c_0 + l x$ for any positive integer $l$.

If either $d_E$ or $c_E$ is not positive, then we get such $d_0$
and $c_0$ by adding $(x-y)$ and $x$ to $d_E$ and $c_E$, respectively. Indeed, since $|d_E| < x-y$, there is $d_E + x - y > 0$,
and the analogue holds for $c_E$.

\end{proof}

Notice that all coprime pairs $(x,y)$ can be generated in a very efficient way,
using the concept of the Stern-Brocot tree. Indeed, with adequate data structures,
it is possible to generate unique coprime pairs, using constant time per pair.
See, for example, Chapter 4 of \cite{GKP} for details.

With a similar reasoning, we can obtain the following result based on $c$ and $d$.

\begin{proposition}\label{cd_existence}
For any integer pair $(c,d)$ with $0<d<c$ and $\gcd(c,d)=1$, there exist infinitely
many positive integer pairs $(x,y)$ such that there exists an $(a,b)$-regular graph
$G$ of order $(n,m)$, where $n = x c$, $m = y c$, $a = y d$, $b = x d$, with $b=n-m+1$.
All such $(x_l,y_l)$ pairs are coprime and can be computed as $(x_l,y_l)=(x_l,x_l-z_l)$,
with $(x_l,z_l) = (x_0 + l c, z_0 + l d)$ for any positive integer $l$,
where $x_0$ and $(-z_0)$ are B\'ezout coefficients for $d$ and $c$ in $dx + c(-z) = 1$.
\end{proposition}

Now, let us analyze constructions based on the value of $m$.

\begin{proposition}\label{m_existence}
Consider an integer $m$, $m \geq 3$. Let $c y$ be any positive integer factorization
of $m$ with $c \geq 2$, and let $p x$ be any positive integer factorization of
$m-1$, with $c>p$ and $x>y$. Then there exists an $(a,b)$-regular graph $G$ of order $(n,m)$,
where $m = c y$, $n = c x$, $d = c - p$, $a = y d$, $b = x d$, with $b=n-m+1$.
\end{proposition}
\begin{proof}
Let $c y$ and $p x$ be such factorizations. Notice that there always exists at least one such a pair of factorizations,
namely $c=m$, $y=1$, $p=1$, and $x=m-1$. Using the formulas $n = c x$, $d = c - p$, $a = y d$, $b = x d$, we get integer values for $n$, $d$, $a$, and $b$. So, like in the proof of Proposition \ref{m_existence}, the thesis is equivalent to satisfying $b = n - m + 1$. Indeed, by rewriting, we have $b = x d = x (c - p) = x c - x p = n - m + 1$.
\end{proof}

Notice that the function that assigns to every integer $m$ the cardinality of
the set of values $n$, such that $n>m>1$ and there exists an $(a,b)$-regular graph
of order $(n,m)$, with $b=n-m+1$, is very difficult to analyze.

On one hand, consider for example a Sophie Germain prime $\alpha$. Recall that a prime number
$\alpha$ is a Sophie Germain prime if $2\alpha+1$ is also a prime.
Let $m=\beta=2\alpha+1$ and let $m-1=2\alpha$.
Then there exists only one factorization $m=cy$ with $c>1$: $y=1$ and $c=\beta$;
and three factorizations $m-1 = p x$ with $x > 1$: $(x = 2, p=\alpha)$, $(x = \alpha, p=2)$,
and $(x = 2\alpha, p=1)$; hence there are only three values of $n$
such that there exists an $(a,b)$-regular graph of order $(n,m)$, with $b=n-m+1$.

It is conjectured that there exist infinitely many Sophie Germain primes,
but as for now we have that the Chen's method proves there are infinitely
many prime numbers $\alpha$, such that either $\beta = 2\alpha + 1$ is a prime or
a product of two distinct prime numbers (see \cite{G11}).
So there are infinitely many values $m$ for which there are at most $9$ values
of $n$ such that there exists an $(a,b)$-regular graph of order $(n,m)$, with $b=n-m+1$.

On the other hand, by the result of Ramanujan \cite{R15}, there are infinitely
many integers $m$ for which the number of distinct divisors $d(m)$ of $m$ satisfies
$d(m) > 2^{\frac{\log(m)}{\log\log(m)}+O(\frac{\log(m)}{(\log\log(m))^2})}$.
So the number of values of $n$ such that there exists an $(a,b)$-regular graph
of order $(n,m)$, with $b=n-m+1$, for such $m$ is very high.

With a similar reasoning as the one for $m$, we obtain an analysis of constructions
based on $a$, $b$, and $n$.

\begin{proposition}\label{a_existence}
Consider an integer $a$, $a \geq 2$. Let $d y$ be any positive integer factorization
of $a$, and let $p z$ be any positive integer factorization of
$a-1$. Then there exists an $(a,b)$-regular graph $G$ of order $(n,m)$,
where $a = d y$, $x = z + y$, $b = d x$, $c = d + p$, $n = c x$, and $m = c y$, with $b=n-m+1$.
\end{proposition}

\begin{proposition}\label{b_existence}
Consider an integer $b$, $b \geq 3$. Let $dx$ be any positive integer factorization of $b$,
with $x \geq 2$, and let $cz$ be any positive integer factorization of $b-1$,
with $c > d$ and $z<x$. Then there exists an $(a,b)$-regular graph $G$ of order $(n,m)$,
where $b=dx$, $y = x - z$, $a=dy$, $m = c y$, $n = c x$, with $b=n-m+1$.
\end{proposition}

\begin{proposition}\label{n_existence_1}
Consider a composite integer $n$, $n \geq 4$. Let $cx$ be any positive integer factorization of $n$,
with $x \geq 2$ and $c \geq 2$, and let $yz$ be any positive integer factorization of $n+1$,
with $z \geq c + 1$. Then there exists an $(a,b)$-regular graph $G$ of order $(n,m)$,
where $n=cx$, $m=cy$, $d=z-c$, $a=dy$, $b=dx$, with $b=n-m+1$.
\end{proposition}

Constructions based on $n$ can also be analyzed from a different point of view.

\begin{proposition}\label{n_existence_2}
Consider a composite integer $n$, $n \geq 4$. Let $cx$ be any positive integer factorization of $n$, with $c \geq 2$ and $x \geq 2$. Then there exists exactly one integer pair $(y,d)$ such that there exists an $(a,b)$-regular graph $G$
of order $(n,m)$, where $m = c y$, $n = c x$, $a = y d$, and $b = x d$, with $b=n-m+1$. Moreover, 
$d$ and $-(x-y)$ are the B\'ezout's coefficients for $x$ and $c$ in $dx - (x-y)c = 1$ with $d < c$ and $x-y<x$.
\end{proposition}

\begin{proof}
First notice that the condition for $n$ to be composite is necessary. Indeed, for
a prime $n$ there is no factorization $n=cx$ with $c \geq 2$ and $x \geq 2$.

Let $c x$ be a positive integer factorization of $n$. Given an integer pair $(y,d)$, applying the formulas $m = c y$, $a = y d$, and $b = x d$, gives us integer values of $n$, $m$, $a$, $b$. So, like in the proof of Proposition \ref{m_existence}, the thesis is equivalent to satisfying $b = n - m + 1$. And, by rewriting $b = n - m + 1$, we obtain $x d = c x - c y + 1$, and $1 = d x - c (x-y)$. 

Recall that, by Lemma \ref{params_properties}, there must be $d < c$ and $x-y < x$.

Let us define $z=(x-y)$ and write $1 = dx - zc$. Since $c>1$, by Theorem \ref{bezout}, the Extended Euclidean Algorithm for $x$ and $c$ gives the B\'ezout's coefficients $d_E$ and $(-z_E)$ such that $|d_E| < c$ and $|-z_E| < x$.

Notice that, since $x \geq 2$ and $c \geq 2$, there must be both $d_E \neq 0$ and $z_E \neq 0$. Moreover, B\'ezout's coefficients for a pair of positive integers, if they are both non-zero, have to be of opposite signs. So $d_E$ and $z_E$ are either both positive or both negative. 

If they are both negative, then adding $c$ and $x$ to $d_E$ and $z_E$, respectively, by an argument similar to that of the proof of Proposition \ref{xy_existence}, gives us the values as needed. If they are both positive, then they satisfy our conditions, and adding $c$ and $x$ gives values that exceed $c$ and $x$, respectively. Hence the values of $(d,z)$ that let us satisfy $b = n - m + 1$, and thus the values of $(y,d)$ are unique.
\end{proof}

Notice that by Lemmas \ref{le:xy}, \ref{params_properties}, and \ref{cd_coprime}, Propositions \ref{xy_existence} - \ref{n_existence_2} cover all  possible cases where  $(a,b)$-regular bipartite graphs of order $(n,m)$, with $1 < m < n$ and $b=n-m+1$, might exist.

Notice also that, since the Extended Euclidean Algorithm runs in time polynomial in the input size (see, for example, \cite{K14} for a detailed analysis), the computations of the first parameter values described in Propositions \ref{xy_existence} and \ref{cd_existence} can be done in polynomial time, and then each following pair can be computed in polynomial time too. Finally, the time complexity of computing the values presented in Proposition \ref{n_existence_2} is polynomial.

The computation time of the values presented in Propositions \ref{m_existence}, \ref{a_existence}, and \ref{b_existence} depends on the number of factorizations, which may be exponential in the input size. The time complexity of finding an integer factorization is still open for the classical model of computation, but can be done in polynomial time on a quantum computer \cite{S99}.

\section{ $k$-critical bipartite biregular graphs}\label{sec:k-critical}

Let us define a special class of $(a,b)$-regular bipartite graphs of order $(n.m)$ that will be the main object of study in this section. Throughout, $n,m,k$ will be positive integers such that $1<m<n$, $k=n-m$, $b = n-m+1$, and $a=\frac{m (n-m+1)}{n}$ is integer.

\begin{definition}
Let $n,m,k$ be positive integers such that $1<m<n$, $k=n-m$, $b = n-m+1$, and $a=\frac{m (n-m+1)}{n}$ is integer. An $(a,b)$-regular graph $G$ of order $(n,m)$ such that, for any $U' \subset U$ with $|U'| = m$, the subgraph $H := G[U', V]$ induced in $G$ by vertex set $(U', V)$ has a perfect matching, is called biregular $k$-critical bipartite.
\end{definition}

Our proofs related to the existence of perfect matchings will be based on the well known Hall's theorem.
\begin{theorem}[Hall's Theorem]
A bipartite graph $G=(U,V;E)$ contains a complete matching from $U$ to $V$ if and only if it satisfies $|N(A)| \geq |A|$ for every $A\subset U$.\label{Hall}
\end{theorem}

Let us first show some result for graphs given by Construction~\ref{cons1}.

\begin{observation}\label{kcons1}
Let $n,m,k$ be positive integers such that $1<m<n$, $k=n-m$, $b = n-m+1$, and $a=\frac{m (n-m+1)}{n}$ is integer.
The graph $G_1$  given in Construction~\ref{cons1} is biregular $k$-critical bipartite if and only if $c=\gcd(n,m)=m$.
\end{observation}
\begin{proof}
Let $G_1$  be a graph given by Construction~\ref{cons1}. As usual, let $d = \gcd(a,b)$. Note that, for any
$v_{j_1},v_{j_2}\in V$, there is either $N(v_{j_1})\cap N(v_{j_2})=\emptyset$ or  $N(v_{j_1})= N(v_{j_2})$.

Suppose first that $c=m$, then recall that then, by Lemma \ref{le:xy}, we have $d=a$. We will show now that, for
any $U' \subset U$ such that $|U'| = m$, the subgraph $H := G[U', V]$
induced in $G$ by the vertex set $(U', V)$ has a perfect matching.

Let $U^{\prime}\subset U$, $|U^{\prime}|=m$. We need to show that, for any
$A$, $A \subset U^{\prime}$, $|A|=r$, $A=\{u_{i_0},u_{i_1},\ldots,u_{i_{r-1}}\}$,
the neighborhood $N(A)$, $N(A) \subset V$, satisfies $|N(A)|\geq |A| = r$.
Suppose that  $za<|A|\leq (z+1)a$ for some $z\in[k]$, then, by Pigeonhole Principle, there have to exist pairwise distinct vertices  $v_{j_1},v_{j_2},\ldots,v_{j_z}\in N(A)\subset V$ such that $N(v_{j_s})\neq N(v_{j_t})$ for any $j_s\neq j_t$. Hence $N(A)\geq za$.

Assume now that $c<m$, then let $B=\{v_{c-1,0},v_{c-1,1}\ldots,v_{c-1,y-1}\}$. By Construction~\ref{cons1}, there is
$$N(B)=\{u_{c-d,0}, u_{c-d,1},\ldots,u_{c-d,x-1}, \ldots,u_{c-1,0}, u_{c-1,1},\ldots,u_{c-1,x-1}\}.$$


Hence $|N(B)|=d\cdot b/d=b=n-m+1$. Let $A=(U\setminus N(B))\cup \{u_{c-1,0}\}$. Observe that then, in the subgraph $H=G[A,V]$, we have $|B|>1$ and $|N_H(B)|=1$. Thus, Hall's Theorem implies that $H$ does not have a perfect matching.
\end{proof}

Let us present now our main negative result.

\begin{theorem}\label{thm:main_negative}
Let $n,m,k$ be positive integers such that $1<m<n$, $k=n-m$, $b = n-m+1$, and $a=\frac{m (n-m+1)}{n}$ is integer.
There exists an $(a,b)$-regular graph of order $(n,m)$ that is not biregular $k$-critical bipartite if and only if $a<m-1$.
\end{theorem}
\begin{proof}

Let $G$ be an $(a,b)$-regular graph of order $(n,m)$. As usual, let $c = \gcd(m,n)$ and $d = \gcd(a,b)$. By Lemma \ref{params_properties}, we have $a < m$.


Suppose that $a=m-1$. Then, by simple rewriting, we obtain that $b=(m-1)^2$ and $n=(m-1)m$.
Take any $U'$, $U' \subset U$ with $|U'| = m$.  Suppose that the subgraph $H := G[U', V]$
induced in $G$ by the vertex set $(U', V)$ does not have a perfect matching.
Observe that $|N(U')|\geq m-1$, because deg$(u)=m-1$ for any $u\in U'$. Therefore,
by Hall's Theorem, we can assume that there exists $A \subseteq U'$ with $|A|=m$ and
$|N(A)|=m-1$. It implies that there exists a vertex $v\in V$ such that $N(v)\cap A=\emptyset$.
Hence deg$(v)\leq |U|-|A|=(m-2)m<(m-1)^2=b=$deg$(v)$ - a contradiction.

From now on, we assume that $a<m-1$. If $c<m$, then Observation~\ref{kcons1} implies that
the $(a,b)$-regular graph of order $(n,m)$ given by Construction~\ref{cons1} is not biregular $k$-critical bipartite.\\

For $c=m$, let $G^{\prime}$ be an $(a,b)$-regular bipartite graph with color classes $U=\{u_0,u_1,\ldots,u_{n-1}\}$ and $V=\{v_0,v_1,\ldots,v_{m-1}\}$ such that $E(G^{\prime})=\{(u_{(i+\alpha+zm)\bmod{n}},v_i) \mid i \in [m], \alpha\in[a],z\in[x]\}$. Note that $N_{G^{\prime}}(v_0)\cap N_{G^{\prime}}(v_{m-1})=\emptyset$, since $a<m-1$.\\
We will now construct the graph $G$ by exchanging some edges in $G^{\prime}$. Namely, let $V(G)=V(G^{\prime})$ and
$$E(G)=(E(G^{\prime})\setminus\{(u_{(a+zm)\bmod{n}},v_{1}),(u_{zm\bmod{n}},v_{m-1})\mid  z\in[x]\})$$ $$\cup\{(u_{zm\bmod{n}},v_{1}),(u_{(a+zm)\bmod{n}},v_{m-1})\mid  z\in[x]\}.$$ Let $B=\{v_{0},v_{m-1}\}$. By construction, there is
$N(B)=\{u_{(\alpha+zm)\bmod{n}},\;$ for $\alpha\in[a], z\in[x]\}.$


Hence $|N(B)|=a\cdot x=b=n-m+1$. Let $A=(U\setminus N(B))\cup \{u_{0}\}$. Observe that then, in the subgraph $H=G[A,V]$, we have $|B|=2$ whereas $|N_H(B)|=1$. Thus Hall's Theorem implies that $H$ does not have a perfect matching.
\end{proof}

Let us now present the main result.

\begin{theorem}\label{thm:main_positive}
Let $n,m,k$ be positive integers such that $1<m<n$, $k=n-m$,  $a=\frac{m (n-m+1)}{n}$ is integer, and $b=n-m+1$.
Then there exists an $(a,b)$-regular graph of order $(n,m)$ that is biregular $k$-critical bipartite.
\label{positive_const_biregular}
\end{theorem}

\begin{proof}

Recall that $c = \gcd(n,m)$, $d = \gcd(a,b)$, $n = x c$, $m = y c$, $a = y d$, $b = x d$. Let $G_2$  be the graph given by Construction~\ref{cons3}, i.e., a graph $(U,V; E)$, with $U=\{u_i \mid i \in [n]\}$, $V=\{v_j \mid j \in [m]\}$, $E = \{ (u_i, v_{(j+\alpha) \bmod{m}}) \mid i\in [n], \alpha \in [a], j=\ceil{\frac{i y}{x}}\}$.

Consider the cardinalities $|\{i \mid i\in [n], \ceil{\frac{i y}{x}} = j \pmod{m} \}|$ for $j \in [m]$.
Recall that, by Lemma \ref{number_of_i_1}, when $y=1$, then all cardinalities are equal $\ceil{\frac{x}{y}}=\floor{\frac{x}{y}}=x$. Otherwise,
since $x$ and $y$ are coprime, $(\ceil{\frac{x}{y}} - \frac{x}{y})y = (y - x \bmod y) \leq y - 1$, and the cardinalities alternate
between $\floor{\frac{x}{y}}$ and $\ceil{\frac{x}{y}}$.


We will show that, for any $U^{\prime} \subset U$ such that $|U^{\prime}| = m$, in the subgraph $H := G[U^{\prime}, V]$
induced in $G$ by the vertex set $(U^{\prime}, V)$, there exists a complete matching of size $|U^{\prime}|$.
By Theorem \ref{Hall}, it means that, for any $A$, $A \subset U$, $|A|=r\leq m$,
$A=\{u_{i_0},u_{i_1},\ldots,u_{i_{r-1}}\}$, the neighborhood
$$N(A)=\left\{v_{\left(\ceil{\frac{i_{s} y}{x}}+\alpha\right) \bmod{m}}: \,\alpha \in [a],\; s \in [r]\right\}\subset V,$$
satisfies $|N(A)|\geq |A|$. 

Given such an $A$, let $B(A) = \{v_j \mid \ceil{\frac{i_{s} y}{x}} \bmod{m} = j, s \in [r] \}\subset N(A)$, i.e., $B(A)$ is the set of the ``first neighbours'' of vertices from $A$.




For the argument, let us consider the elements of $V$ put on a directed cycle $C_m$,
where there is an arc $(v_{j_1}, v_{j_2})$ if and only if $j_2 - j_1 = 1 \pmod{m}$. We will analyze
the relative positions of the elements of $B(A)$ on $C_m$.

Suppose that there exists a subset $A$, $A \subset U$, such that $|N(A)|/|A| = t < 1$. Among all such subsets, let $A$ be such that the value of $t$ is the lowest possible, the minimum length of a path that covers all vertices of $B(A)$ in $C_m$ is the lowest possible, and such a path terminates with the largest possible number of vertices $v_j$ for which $ |\{ i \mid i\in [n], \ceil{\frac{i y}{x}} = j \pmod{m} \}| = \ceil{\frac{x}{y}}$. Let $P$ denote such a path (if there are more than one, we chose arbitrarily one of them) .


Clearly, $N(A)$ does not cover the whole cycle $C_m$, as this would mean $|N(A)| = m \geq |A|$.

Let us show that $N(A)$ induces a path in $C_m$. Suppose on the contrary that it is not the case, so we can partition $A$ into a family of disjoint sets $A_0,\dots,A_{s-1}$, each one corresponding to an inclusion maximal path $P_{A_l}$, $l\in [s]$, of the subgraph induced by $N(A)$ in $C_m$. Note that the length of each path $P_{A_l}$, $l\in [s]$, is strictly smaller than the length of $P$. Suppose that we have $|N(A_l)|/|A_l| > |N(A)|/|A|$, for all $l\in [s]$. Thus, there is $|A||N(A_l)|>|A_l| |N(A)|$, for all $l\in [s]$, which implies that $|A|(|N(A_0)|+\ldots+ |N(A_{s-1})|)>(|A_0|+\ldots+|A_{s-1}|) |N(A)|$. Since the sets $A_l$, $l\in [s]$, are pairwise disjoint, and so are their neighborhoods, this leads to $|A||N(A)|>|A| |N(A)|$ - a contradiction. So, for at least for one $l \in [s]$, there is $|N(A_l)|/|A_l| \leq |N(A)|/|A|$ - a contradiction with the choice of $A$.

Now, given that $N(A)$ induces a path in $C_m$, we can also show that $B(A)$ induces a path in $C_m$. Notice that, since $N(A)$ induces a path in $C_m$, there must be $V(P) \subseteq N(A)$. Suppose that $P$ contains a vertex $v_j$ not in $B(A)$. This means that $\{ i \mid \ceil{\frac{i y}{x}} = j \pmod{m} \} \cap A = \emptyset$, but $v_j$ is in $N(A)$. So we can create $A^{\prime} = A \cup \{u_i\}$, where $\ceil{\frac{i y}{x}} = j \pmod{m}$, with $|A^{\prime}|>|A|$ and $|N(A^{\prime})| = |N(A)|$ - a contradiction with the choice of $A$, unless $|A|=m$. But $|A|=m$ can not hold, since it would mean that $|N(A)| \leq m-1$, and, for any $v_j \in V$, there are only $m-1$ non-neighbors of $v_j$ in $U$. So, additionally we have shown that we may assume $|A| \leq m-1$ and $|N(A)|\leq m-2$.

In a similar way, we can prove that $A = \{ u_i \mid \ceil{\frac{i y}{x}} = j \pmod{m}, v_j\in B(A) \}$.
In other words, $A$ corresponds to the set of solutions to $\ceil{\frac{i y}{x}} = j \pmod{m}$
for an interval $J$ of consecutive values of $j$.



Therefore, we can assume 
$|B|=q y + \gamma$ for $\gamma\in[y]$. By Lemma~\ref{number_of_i_1},
the $qy$ vertices in $B$ correspond to $qx$ vertices in $A$. The remaining $\gamma$ vertices in
$B$ can be partitioned into $\gamma'$ vertices that correspond to $\gamma' \ceil{\frac{x}{y}}$ vertices
in $A$ and $\gamma''$ vertices that correspond to $\gamma'' \floor{\frac{x}{y}}$ vertices in $A$.
Again by Lemma~\ref{number_of_i_1}, we have $\gamma'+\gamma''=\gamma$, $\gamma' \leq (x \bmod y)$ and $\gamma'' \leq (y - x \bmod y)$.
Hence $qx\leq|A|\leq  qx+\gamma'\ceil{\frac{x}{y}}+\gamma''\floor{\frac{x}{y}}$ and $|N(A)|=qy+\gamma+dy-1$. \\

By the choice of $A$,
$$|N(A)|=qy+\gamma'+\gamma''+dy-1 <|A|\leq qx+\gamma'\ceil{\frac{x}{y}}+\gamma''\floor{\frac{x}{y}}$$.

Suppose $\gamma = 0$, we have $|A|=qx$ and the inequality reduces to $qy + dy - 1 < qx$, which is equivalent
to $q > \frac{dy-1}{x-y}$. By Lemma \ref{params_properties}, we have $\frac{yd-1}{x-y} = \frac{m-1}{x}$,
and so it is equivalent to $q > \frac{m-1}{x}$. We get $|A| = qx > m-1$ - a contradiction, since we already
have already showed that $|A| \leq m-1$. So we may assume that $\gamma'+\gamma'' > 0$.\\

Suppose that $(q+1)x\leq m-1$, then  $q+1\leq\frac{m-1}{x}=\frac{yd-1}{x-y}$ by Lemma \ref{params_properties}. Hence
$$qy+y +yd-1-(y-\gamma'-\gamma'') \geq qx+x-(y-\gamma'-\gamma'')$$
Recall that $\gamma' \leq (x \bmod y)$ and $\gamma'' \leq (y - x \bmod y)$, $(x \bmod y)\ceil{\frac{x}{y}}+(y - x \bmod y)\floor{\frac{x}{y}}=x$ and $\frac{x}{y}>1$. Therefore
$$qx+x-(y-\gamma'-\gamma'') = qx+x-(x \bmod y-\gamma')-((y - x \bmod y)-\gamma'')$$

$$qx+x-x \bmod y+\gamma'-(y - x \bmod y)+\gamma'')\geq
qx+x-x \bmod y+\gamma'\ceil{\frac{x}{y}} - (y - x \bmod y)+\gamma''\floor{\frac{x}{y}}$$

$$qx+x- (x \bmod y -\gamma')\ceil{\frac{x}{y}} - ((y - x \bmod y)-\gamma'')\floor{\frac{x}{y}}=qx+\gamma'\ceil{\frac{x}{y}} +\gamma''\floor{\frac{x}{y}}.$$

Thus
$$|N(A)|=qy+\gamma'+\gamma''+dy-1 \geq qx+\gamma'\ceil{\frac{x}{y}}+\gamma''\floor{\frac{x}{y}}\geq|A|$$
- a contradiction with the choice of $A$, so we may assume that $(q+1)x \geq m$.

On the one hand, $m\leq (q+1)x$ implies $\frac{m-1}{x} < q+1$. Since $p=\frac{m-1}{x}$, it means $q \geq p$.
On the other hand, recall that $|A| \leq m-1$. So $|A| > |N(A)| = qy+\gamma+dy-1$ implies $m-1 > qy+\gamma+dy-1$.
By Lemma \ref{params_properties}, $qy+\gamma+dy-1 = qy + \gamma + p(x-y) = qy + \gamma + m - 1 - py$.
So we obtain $y(p-q) \geq \gamma$. Since $\gamma > 0$, there is $q < p$. So we reach a contradiction.
\end{proof}

It is interesting to note how a little difference between Construction \ref{cons1} and
Construction \ref{cons3} changes the properties of the resulting graphs with respect
to the property of being $k$-critical bipartite. Indeed, the only difference is in changing the first vertex adjacent
to any $u_i \in U$ from $(\floor{\frac{i}{x}} y \bmod{m})$ to $(\ceil{\frac{iy}{x}} \bmod{m})$.
In fact, by reasons similar to what we presented for Construction \ref{cons1},
replacing $(\floor{\frac{i}{x}}y \bmod{m})$ with $(\ceil{\frac{i}{x}}y \bmod{m})$ does not
change the properties of the resulting graph with respect to M$k$CBG-$(n,m)$.
On the other hand, replacing $(\ceil{\frac{iy}{x}} \bmod{m})$ with $(\floor{\frac{iy}{x}} \bmod{m})$
in Construction \ref{cons3} also generates graphs that are solutions to M$k$CBG-$(n,m)$.

As an interesting generalization of Construction \ref{cons3},
if the neighbors following $v_j$, $j=\ceil{\frac{i y}{x}}$,
instead of being consecutive ($s=1$), are separated by any $s$, a divisor of $x$,
the resulting construction is also a biregular $k$-critical bipartite graph. In other words,
for any positive integers $n,m$ such that $1<m<n$, $a=\frac{m (n-m+1)}{n}$ is integer, and $b=m-m+1$,
$G=(U,V; E)$, with $U=\{u_i \mid i \in [n]\}$, $V=\{v_j \mid j \in [m]\}$,
$E = \{ (u_i, v_{(j+s\alpha) \bmod{m}}) \mid \alpha \in [a], j=\ceil{\frac{i y}{x}}\}$,
for any $s$ such that $x = 0 \pmod s $, also is $k$-critical bipartite graph. On the other hand,
for any $s$, such that $x \neq 0 \pmod s $, the resulting graph is in general not even
biregular.

\section{Conclusions}\label{sec:conclusions}

We define the Minimum $k$-Critical Bipartite Graph problem for $(n,m)$ : to find a bipartite graph $G=(U,V;E)$,
with $|U|=n$, $|V|=m$, $k=n-m$, that is a $k$-critical bipartite graph, and the tuple $(|E|, \Delta_U, \Delta_V)$,
where $\Delta_U$ and $\Delta_V$ are the maximum degree in $U$ and $V$, respectively,
is lexicographically minimum.

We study it in the case of unbalanced biregular graphs, i.e., when $n,m,k$ are positive
integers such that $1<m<n$, $k=n-m$, and $a=\frac{m (n-m+1)}{n}$ is integer, and $b=m-m+1$.
We show that if $a=m-1$, then all $(a,b)$-regular bipartite graphs of order $(n,m)$
are $k$-critical bipartite, and for $a<m-1$, it is not the case.
We characterize the values of $n$, $m$, $a$, and $b$ that admit an $(a,b)$-regular
bipartite graph of order $(n,m)$, and give a simple construction that creates such
a $k$-critical bipartite graph whenever possible. Our analysis leads to simple algorithmic recipes
that can be exploited for generating biregular $k$-critical bipartite graphs.

We hope that our results will motivate further studies, in particular in relation
with the topics that are mentioned in the introduction.

At the moment, we are working on the following conjecture that generalizes
Construction \ref{cons3} to unbalanced bipartite graphs that are not biregular.

\begin{conjecture}
Let $n,m,k$ be positive integers such that $1<m<n$, $k=n-m$, and $a=\frac{m (n-m+1)}{n}$ is not integer.
Let $a'= \ceil{a}$. Then the graph $(U,V; E)$, with $U=\{u_i \mid i \in [n]\}$, $V=\{v_j \mid j \in [m]\}$,
$E = \{ (u_i, v_{(j+\alpha) \bmod{m}}) \mid i \in [n], \alpha \in [a'], j=\ceil{\frac{i m}{n}}\}$,
is $k$-critical bipartite.\label{irreg}
\end{conjecture}

\bibliographystyle{acm}
\bibliography{construction_paper}

\end{document}